\documentclass{amsart}
\usepackage[fontsize=10pt]{scrextend}
\usepackage{amsmath}
\usepackage{amssymb}
\numberwithin{equation}{section}
\usepackage[english]{babel}
\usepackage{amsthm}
\usepackage{bbm}
\usepackage{amsfonts}
\usepackage{comment}
\usepackage{mathrsfs}
\usepackage{stix}
\usepackage[margin=3cm]{geometry}
\usepackage[textsize=scriptsize]{todonotes}
\usepackage{graphicx, tikz}
\usepackage{float}
\usepackage{xcolor}
\usepackage{bbm}
\usepackage[format=plain,
            font=it]{caption}
\usepackage{hyperref}
\usepackage{graphicx, caption, subcaption}
\hypersetup{
    colorlinks=true,
    linkcolor=blue,
    citecolor=red,
    filecolor=magenta,      
    urlcolor=black,
    pdftitle={Topologically protected modes in systems with damping},
    pdfpagemode=FullScreen,
    }

\newtheorem{theorem}{Theorem}[section]
\newtheorem{lemma}[theorem]{Lemma}
\newtheorem{corollary}[theorem]{Corollary}

\theoremstyle{definition}
\newtheorem{definition}[theorem]{Definition}

\theoremstyle{remark}
\newtheorem{remark}[theorem]{Remark}

\numberwithin{equation}{section}

\newcommand{\bdis}{\begin{displaymath}}   \newcommand{\edis}{\end{displaymath}}
\newcommand{\beq}{\begin{equation}}       \newcommand{\eeq}{\end{equation}}
\newcommand{\beql}[1]{\begin{equation}\label{#1}} 
\newcommand{\beqla}{\begin{equation}\label}
\newcommand{\beqnt}{\begin{equation*}}    \newcommand{\eeqnt}{\end{equation*}}
\newcommand{\bsplit}{\begin{split}}       \newcommand{\esplit}{\end{split}}
\newcommand{\bspl}{\begin{split}}          \newcommand{\espl}{\end{split}}
\newcommand{\bproof}{\begin{proof}}       \newcommand{\eproof}{\end{proof}}
\newcommand{\bcenter}{\begin{center}}     \newcommand{\ecenter}{\end{center}}
\newcommand{\bce}{\begin{center}}         \newcommand{\ece}{\end{center}}

\newcommand{\ve}{\vert }

    \font\bb=msbm10

\font\elevenrm=cmr11     \font\twelverm=cmr12
\def\1{\hbox{\elevenrm 1}\!\!\hbox{\twelverm 1}}

\newcommand{\upd}{\mathrm{d}}

\def\C{\hbox{\bb C}}    
\def\R{\hbox{\bb R}}   

     \def\d{\delta}          
    \def\ve{\varepsilon}    
               
\def\k{\kappa }          \def\l{\lambda }    
    \def\m{\mu}                      
    \def\r{\rho}                 
                   
   \def\f{\phi }                
                         
       \def\w{\omega }

\begin{document}

\title{Topological interface modes in systems with damping}

\author[Konstantinos Alexopoulos]{Konstantinos Alexopoulos}
\address{\parbox{\linewidth}{Konstantinos Alexopoulos\\Centre de Mathématiques Appliquées, École Polytechnique, 91120 Palaiseau, France}}
\email{konstantinos.alexopoulos@polytechnique.edu}

\author[Bryn Davies]{Bryn Davies}
\address{\parbox{\linewidth}{Bryn Davies\\Mathematics Institute, University of Warwick, Coventry CV4 7AL, UK}}
\email{bryn.davies@warwick.ac.uk}

\author[Erik Orvehed Hiltunen]{Erik Orvehed Hiltunen}
\address{\parbox{\linewidth}{Erik Orvehed Hiltunen\\Department of Mathematics,  University of Oslo, Moltke Moes vei 35, 0851 Oslo, Norway}}
\email{erikhilt@math.uio.no}

\begin{abstract}
    We extend the theory of topological localised interface  modes to systems with damping. The spectral problem is formulated as a root-finding problem for the interface impedance function and Rouché's theorem is used to track the zeros when damping is introduced. We show that the localised eigenfrequencies, corresponding to interface modes, remain for non-zero dampings. Using the transfer matrix method, we explicitly characterise the decay rate of the interface mode. 
\end{abstract} 

\maketitle



\section{Introduction}

A cornerstone of modern wave physics is the creation of localised eigenmodes, which facilitate the strong focusing of wave energy and are often the starting point for building wave guides and other wave-focusing devices. One way to realise such modes is to add defects to periodic media. When done correctly, this creates eigenmodes that belong to the pure-point spectrum of the operator and are localised such that they decay exponentially as a function of distance from the defect. Further, introducing defects with account of the system's underlying topological properties (taking inspiration from the field of topological insulators \cite{hasan2010colloquium, kane2005z}) has yielded a systematic way to create strongly localised eigenmodes that are, additionally, robust to imperfections. The specific, tunable and controllable nature of these topological localised modes (see \emph{e.g.} \cite{boardman2011active, chaplain2023tunable, chen2016tunable, makwana2018geometrically}) means that waves of selected frequencies can be focused at desired locations. As a result, they are ideally suited to serve as the basis for developing a wide variety of wave guiding and control devices and have been exploited extensively by experimentalists: see \emph{e.g.} \cite{PriceRoadmap} for a review of topological photonics.

The mathematical theory for localised modes has been thoroughly developed for undamped systems recently. For one-dimensional systems, there are extensive theories for both the Schr\"odinger equation \cite{fefferman2014topologically, fefferman2017topologically} and classical wave systems \cite{alexopoulos2023topologically, ammari2020topologically, ammari2022robust, lin2022mathematical, coutant2024surface, craster2023asymptotic, thiang2023bulk}. There have also been important breakthroughs in extending these ideas to multi-dimensional partial differential models \cite{ammari2020topologically, ammari2022robust, bal2023edge, drouot2019bulk, li2024topological, PhysRevA.93.033822, Fefferman_2016}. These approaches typically rely on either the use of integral operators \cite{ammari2022robust, ammari2020topologically} or reducing the problem to (coupled) Dirac operators \cite{fefferman2014wave, drouot2019bulk, ammari2020high}. Nevertheless, the crucial mechanisms can already be studied in one-dimensional systems as it is informative to consider the problem of introducing an interface that breaks a single axis of periodicity.

Real-world physical systems invariably experience some degree of energy dissipation or damping. Damping can be represented in a system as a non-zero imaginary part of the permittivity function. For example, this is common in models for the permittivity of several important materials, such as the behaviour of metals at optical frequencies (described \emph{e.g.} by the Drude model \cite{ordal1983optical}) and popular material choices for the fabrication of photonic crystals (\emph{e.g.} halide perovskites \cite{AM, AD}). However, the existing mathematical theory mostly focuses on systems with real-valued material parameters. 

This work extends the mathematical theory of localised modes to systems with damping in terms of complex-valued permittivity functions. The main theoretical challenge of these settings is that the addition of damping perturbs the real-valued spectrum of the differential operator into the complex plane. We have as our starting point the work carried out in \cite{alexopoulos2023effect, coutant2024surface, lin2022mathematical, thiang2023bulk, tsukerman2023topological, xiao2014surface}, which use surface impedance functions to relate the existence of interface modes to properties of the Floquet-Bloch spectrum of the periodic (bulk) materials. These localised interface modes are also sometimes known as \emph{edge modes} and the theory relating them to bulk properties of the material is known as the \emph{bulk-edge correspondence}. Our approach is to use Rouché's theorem to track how the real-valued eigenfrequencies of interface modes in undamped materials are perturbed in the complex plane as damping is introduced.

We will begin by summarising the main results of this work in Section~\ref{sec:Summary}, along with introducing the key definitions and tools that will be required. The problem will be introduced in more detail in Section~\ref{sec:Mathematical setting}, followed by some basic properties of the impedance functions in Section~\ref{sec:Symmetry results}. In Section~\ref{sec:Undamped}, we will summarise the existing results for undamped systems, in which case the existence and uniqueness of localised interface modes is well understood. Finally, in Section~\ref{sec:Damped}, we move to the case of damped systems with complex eigenfrequencies. We consider this as a perturbative regime of the undamped case and, using Rouch\'e's theorem, we show the existence and (in a suitable sense) uniqueness of interface modes.

In Section \ref{sec:Asymptotic behaviour}, using transfer matrices, we compute an estimate for the exponential decay rate of the localised eigenmodes. In Section \ref{sec:Large damping}, we numerically investigate the conditions of Theorem \ref{thm:CC} and discuss how the result can be used for larger dampings. Finally, in Section \ref{sec:Numerical example}, we consider a numerical example of a damped system. We show how the behaviour of the band gaps changes in the presence of damping and how this affects the location of the frequency of the interface localised mode in a band gap. In addition, we show the localisation of the mode at the interface and its decaying character at infinity.

\section{Summary of results}\label{sec:Summary}
We begin by outlining our method and summarising the main contributions of our work, which is a generalises the theory set out in \cite{coutant2024surface, thiang2023bulk, alexopoulos2023topologically} to damped systems. We will recall details from this previous work, as needed.

We will consider the spectrum of a one-dimensional differential operator posed on a domain that has an interface formed by gluing together two different periodic materials. We will label the two materials by $A$ and $B$ and fix the coordinate system so that the interface falls at the origin, denoted by $x_0$. Let $\ve:\mathbb{R}\to\mathbb{C}$ denote the system's permittivity function whose non-zero imaginary part represents the damping in the system. We define $\ve$ by
\begin{align*}
    \ve(x) = 
    \begin{cases}
        \ve_A(x), \quad x < x_0,\\
        \ve_B(x), \quad x \geq x_0,\\
    \end{cases}
\end{align*}
where $\ve_j$, for $j=A,B$, denotes the permittivity function of material $j$. Suppose that $\ve(x)$ is piecewise smooth and that each $\ve_j$ is periodic with period $1$. We consider the differential eigenvalue problem $\mathcal{L}u=\omega^2u$ for the operator
\begin{equation}\label{def:L:summary}
    \mathcal{L}u := - \frac{1}{\mu_0}\frac{\partial}{\partial x} \left( \frac{1}{\ve(x)} \frac{\partial u}{\partial x} \right),
\end{equation}
where $\m_0\in\mathbb{R}$ denotes the constant magnetic permeability and $\w\in\mathbb{C}$ is the frequency.

We are interested in showing existence of eigensolutions of \eqref{def:L:summary} which are localised around the interface $x_0$, meaning they are non-zero and decay exponentially as $|x|\to\infty$. The starting point for this is to study the Floquet-Bloch spectra of the two periodic materials $A$ and $B$. In each case, we search for eigensolutions that satisfy the Floquet-Bloch quasiperiodic conditions
\begin{align}\label{eq:FB:summary}
    u(x+1) = e^{i\k}u(x)\quad\text{and}\quad \frac{\partial u}{\partial x}(x+1) = e^{i\k}\frac{\partial u}{\partial x}(x),
\end{align}
for some $\k$ in the first Brillouin zone $\mathcal{B}:=[-\pi,\pi]$. This reveals a countable sequence of spectral ``bands'' of solutions. In the case of damping, these bands will be complex-valued. Since we require $\k$ to be real, these solutions must have constant spatial amplitude. As a result, the localised solutions we are looking for will have eigenvalues located in the gaps between these bands. In this setting, a band gap is a connected component of the complement of the Floquet-Bloch spectrum in $\mathbb{C}$. The two materials $A$ and $B$ are allowed to have different spectra; however, for an eigenmode to decay in both directions away from the interface, the two materials must have a non-empty common band gap. The localised eigenvalues must fall in this common band gap.

Once a common band gap of the two materials has been identified, the properties of the two materials need to be tuned so that an interface mode is allowed to exist. This is captured by the \emph{surface impedance} function of each material:
\begin{align}\label{def:Z_A and Z_B:summary}
    Z_A(\w) := - \frac{u_{A}(x_0^-,\w)}{\frac{1}{\ve(x_0^-)}\frac{\partial u_{{A}}}{\partial x}(x_0^-,\w)} \quad \text{ and } \quad Z_B(\w) :=  \frac{u_{{B}}(x_0^+,\w)}{\frac{1}{\ve(x_0^+)}\frac{\partial u_{{B}}}{\partial x}(x_0^+,\w)}, \quad \w\in\mathbb{C},
\end{align}
where $u_{j}$ is a decaying solution to each half-space problem $j=A$ and $j=B$. We define the \emph{interface impedance} to be the function $Z:\C\to\C$ given by
\begin{equation} \label{eq:Z_summary}
    Z(\w) = Z_A(\w) + Z_B(\w).
\end{equation}
The interface impedance $Z$ reveals when the half-space solutions can be glued together to create a localised mode: an interface mode exists with eigenfrequency $\omega\in\C$ if and only if 
\begin{align}\label{eq:impedancecondition:summary}
    Z(\w) = 0.
\end{align}
This is a generalisation of the standard approach for undamped materials \cite{alexopoulos2023topologically,coutant2024surface,thiang2023bulk}. In the undamped case, we will write the interface impedance $Z^{(U)}$, while in the damped case we will write $Z^{(D)}$.

The key conclusion of the above is that the existence of localised interface modes is equivalent to a root finding problem for the interface impedance function $Z(\w)$ in  \eqref{eq:Z_summary}. This is a complex-valued function of a complex variable. Our approach to proving the existence of these roots is to study the problem as a perturbation of the system with no damping. From \cite{alexopoulos2023topologically,coutant2024surface,thiang2023bulk}, we have the existence and uniqueness of interface-localised modes with real-valued eigenfrequencies in the case of no damping. In this case, the interface impedance function $Z(\omega)$ is real-valued whenever $\omega$ is real, so it reduces to a root finding problem for a real-valued function of a real variable. When damping is introduced, these roots are perturbed off the real axis into the complex plane. The key result of this work is using Rouch\'e's theorem to keep track of these roots as the damping is introduced. 

We consider a regime of small damping, \emph{i.e.}, we introduce a parameter $\d>0$ such that
\begin{align*}
    \lim_{\d\to0} \sup_{x\in\R} \mathrm{Im}(\ve(x)) = 0.
\end{align*}
The following theorem is the main result of this work. It uses the notation $\mathfrak{A}(\delta)$ for the subset of $\C$ that is the complement of the Floquet-Bloch spectrum of the system with damping parameter $\delta$ (\emph{i.e.} $\mathfrak{A}(\delta)$ is the ``complex band gap'' of the damped material). Hence, $\mathfrak{A}(0)$ is the union of all the band gaps of the undamped system.
\begin{theorem}
    Let $\mathfrak{A}(\delta)$ denote a common band gap of \eqref{def:L:summary}-\eqref{eq:FB:summary}. For each root $\w_U$ of $Z^{(U)}$ in $\mathfrak{A}(0)$, there exists a root $\w_D$ of $Z^{(D)}$ in $\mathfrak{A}(\delta)$, for $0<\d\ll 1$, converging to  $\w_U$ as $\d\to 0$.
\end{theorem}

Observe that the existence of the undamped localized eigenfrequency $\w_U$ was established in \cite{coutant2024surface,thiang2023bulk}. These eigenfrequencies were also shown to be unique inside each (real) band gap. This gives us the existence of interface modes for frequencies which lie in the spectral band gaps of \eqref{def:L:summary}-\eqref{eq:FB:summary} for small but non-zero $\delta$.

Using the transfer matrix method, we also study the spectral properties of the matrices describing the behaviour of the solution as we move further from the interface $x_0$. We obtain an equivalent condition to \eqref{eq:impedancecondition:summary} in terms of the eigenvectors of the transfer matrices and we get the asymptotic behaviour of $u$ as $|x|\to\infty$.

\section{Mathematical setting} \label{sec:Mathematical setting}
We now define the problem outlined above, and introduce the setting to be considered in the remainder of this work.

\subsection{Damped systems}\label{subsection:MS:Damped systems} \hfill

We consider two materials $A$ and $B$. Each one is of the form of a semi-infinite array. The arrays are glued together at the origin and we assume that the material $A$ extends to $-\infty$ and the material $B$ extends to  $+\infty$. In addition, we assume that each material is constructed by repeating periodically a unit cell. We denote the permittivity function of material $A$ by $\ve_A$ and of material $B$ by $\ve_B$. Our assumptions on the permittivity functions are the following: 
\begin{itemize}
    \item piecewise smooth and complex-valued, \emph{i.e.} $\ve_j:\mathbb{R}\to\mathbb{C}$, 
    \item positive real part: $\mathrm{Re}(\ve_j) > 0$,
    \item periodic, \emph{i.e.} $\ve_j(x)=\ve_j(x+1)$, and
    \item inversion symmetric, \emph{i.e.} $\ve_j(x+h)=\ve_j(x+1-h)$, for $h\in(0,1)$,
\end{itemize}
where $j=A,B$.

We define the sequence $\{x_n\}_{n\in\mathbb{Z}}$ to be the set of endpoints of each one of the periodically repeated cells, \emph{i.e.} $x_n=x_0+n$ for $n\in\mathbb{Z}$. We take $x_0$ to be the interface point of the two materials, and so, for $n<0$, $x_n$ is in material $A$ and for $n>0$, $x_n$ is in material $B$.

\subsection{Differential problem}\hfill

The differential problem we are studying in our search for localised eigenmodes is the eigenvalue problem
\begin{align}\label{eq:Helmholtz}
    \mathcal{L}u=\omega^2u
\end{align}
with
\begin{equation}\label{def:L}
    \mathcal{L}u := - \frac{1}{\mu_0}\frac{\partial }{\partial x} \left( \frac{1}{\ve(x)} \frac{\partial u}{\partial x} \right),
\end{equation}
where $\m_0\in\mathbb{R}_{>0}$ is the magnetic permeability, which is constant,  $\ve$ is the permittivity of the system given by
\begin{align}\label{def: system perm}
    \ve(x) := 
    \begin{cases}
        \ve_A(x), \quad x < x_0,\\
        \ve_B(x), \quad x \geq x_0,
    \end{cases}
\end{align}
and $\w\in\mathbb{C}$ is the frequency.  Our goal is to find eigenvalues of \eqref{def:L} for which the associated eigenmodes are localised in a neighborhood of the interface $x_0$ and decay (exponentially) as $|x|\rightarrow\infty$.

To find candidate eigenvalues $\omega^2$ for which localised eigenmodes can occur, it is valuable to consider the Floquet-Bloch spectrum of the operators associated to the periodic materials $A$ and $B$. This is the set of Bloch modes which satisfy \begin{equation}\label{eq:Lj}
\mathcal{L}_ju = \omega^2u,
\end{equation}
where the operator $\mathcal{L}_j$ is defined by $\mathcal{L}_j = - \frac{1}{\mu_0}\frac{\partial }{\partial x} \left( \frac{1}{\ve_j(x)} \frac{\partial }{\partial x} \right)$, for $j=A,B$, along with the quasi-periodicity conditions
\begin{align}\label{eq:FB}
    u(x+1) = e^{i\k}u(x)\quad\text{and}\quad \frac{\partial u}{\partial x}(x+1) = e^{i\k}\frac{\partial u}{\partial x}(x),
\end{align}
for some $\k\in [-\pi,\pi]$. Hence, these eigenmodes belong to the space of functions
\begin{align*}
    H^2_{\k} := \Big\{ f\in H^2_{\mathrm{loc}}: \ f(x+1) = e^{i\k} f(x), \ x\in\mathbb{R} \Big\}.
\end{align*}

With this formulation of the problem, we observe that for an eigenmode to be localised at the interface $x_0$, the associated eigenvalue has to lie in the spectral gaps of the problem \eqref{eq:Lj}-\eqref{eq:FB}. This holds since, for an eigenvalue in a spectral band, \eqref{eq:FB} gives that the magnitude of the associated eigenmode does not decay as $|x|\to \infty$. As a result, we are interested in materials $A$ and $B$ with overlapping band gaps, otherwise edge modes cannot exist in this setting.

Finally, we equip the space $H^2_{\k}$ with the standard inner product $\langle \cdot,\cdot \rangle$, given by
\begin{align}\label{def:inner product}
    \langle u,v \rangle = \int_{\mathbb{R}} u(x) \overline{v}(x) \upd x,
\end{align}
for $ u,v\in H^2_{\k}$.

\section{Preliminaries} \label{sec:Symmetry results}


\subsection{Impedance functions}\hfill

We start by defining the impedance function, as described in Section \ref{sec:Summary}.

\begin{definition}[Interface impedance]\label{def: interf imp}
	Let $\w\in\C$ be in a common band gap of $\mathcal{L}_A$ and $\mathcal{L}_B$. We define the interface impedance of the problem \eqref{eq:Helmholtz} by
    \begin{align}\label{def:interface impedance}
        Z(\w) := Z_A(\w) + Z_B(\w),  \quad \w\in\mathbb{C},
    \end{align}
    where $Z_A$ and $Z_B$ denote the surface impedances of materials $A$ and $B$, respectively, given by
    \begin{align}\label{def:surface impedances}
        Z_A(\w)= -\frac{u_{{A}}(x_0^{-},\w)}{\frac{1}{\ve_A(x_0^-)}\frac{\partial}{\partial x}u_{{A}}(x_0^{-},\w)} \ \ \ \text{ and } \ \ \ Z_B(\w)= \frac{u_{{B}}(x_0^{+},\w)}{\frac{1}{\ve_B(x_0^{+})}\frac{\partial}{\partial x}u_{{B}}(x_0^{+},\w)}, \ \ \ \w\in\mathbb{C},
    \end{align}
	where $u_j$, for $j=A,B$, denotes the solution to the problem $\mathcal{L}_j u = \w^2 u$  which decays and has decaying derivative as $x\to-\infty$ for $Z_A$ and as $x\to\infty$ for $Z_B$.
\end{definition}
We observe that the solutions $u_j$, $j=A,B$, exist and are unique (up to scaling)  if and only if $\omega$ is in a common band gap of both $\mathcal{L}_A$ and $\mathcal{L}_B$.

\begin{lemma}\label{lemma:impedance}
Let $\omega \in \mathbb{C}$. A localised interface mode exists at $\omega$ if and only if
\begin{align}\label{eq:impedancecondition}
        Z_A(\w) + Z_B(\w) = 0.
    \end{align}
\end{lemma}
\begin{proof}
	For the proof of this lemma we refer to Theorem 5.2 in \cite{alexopoulos2023topologically}.
\end{proof}

\subsection{Bulk index}\hfill

In the undamped case, existence of interface modes is intimately linked with a topological index of the material known as the bulk index.

\begin{definition}[Bulk index]
	Assume $\mathrm{Im}(\ve) = 0$. Let $\mathfrak{S}=[a,b]$ denote a band gap of $\mathcal{L}_j$ in $\mathbb{R}$ with Bloch mode $u$ at the band edge $a\in\mathcal{B}$. Then we define the associated bulk topological index $\mathcal{J}_{\mathfrak{S}}$ by
    \begin{align}\label{def:bulk index}
        \mathcal{J}_{\mathfrak{S}} :=
        \begin{cases}
            +1, \ \ \ \text{ if } u \text{ is symmetric at $a$},\\
            -1, \ \ \ \text{ if } u \text{ is anti-symmetric at $a$}.
        \end{cases}
    \end{align}
\end{definition}
In the case of damping, it follows from the symmetry of $\ve_j$ that the Bloch mode $u$ at $\kappa = 0$ or $\kappa = \pi$ satisfies 
\begin{equation}
u(-x) = e^{i \phi} u(x),
\end{equation}
for some $\phi \in \mathbb{R}$ known as the \emph{Zak phase}. In the undamped case, the quantity $\phi$ can only attain multiples of $\pi$ and hence, the bulk index is well-defined. However, in the damped case, the Zak phase may attain any real value. As we shall see, despite the fact that the addition of a non-zero imaginary part breaks the quantization of $\phi$, the interface modes persist in the damped case.

\section{Undamped systems}\label{sec:Undamped}

Let us start by considering the undamped case, \emph{i.e.} we are in the following regime
\begin{center}
    $\mathrm{Im}(\ve_A) = \mathrm{Im}(\ve_B) = 0$.
\end{center}
We denote by $Z^{(U)}$ the interface impedance function for the undamped systems, and by $Z^{(U)}_j$ the impedance function of material $j$, with $j=A,B$. First, we will look at the case of $\w\in\mathbb{R}$ and then we will generalise to consider the equivalent root-finding problem over $\w\in\mathbb{C}$.

\subsection{Real frequencies} \label{subsection:RR}\hfill

This setting has been studied extensively in the literature, \emph{e.g.} \cite{alexopoulos2023topologically, coutant2024surface,thiang2023bulk}. We define $\tilde{\mathfrak{A}}^U_{A},\tilde{\mathfrak{A}}^U_{B}\subset \mathbb{R}$ to be band gaps of $\mathcal{L}_A$ and $\mathcal{L}_B$, respectively, such that 
$$
\tilde{\mathfrak{A}}^U := \tilde{\mathfrak{A}}^U_{A} \cap \tilde{\mathfrak{A}}^{U}_{B}\ne\emptyset.
$$ 
Let $\mathcal{J}_{A}$ and $\mathcal{J}_{B}$ be the bulk topological indices associated to the material $A$ and the material $B$, respectively, in $\tilde{\mathfrak{A}}^U$. 

In \emph{e.g.} \cite{alexopoulos2023topologically, coutant2024surface,thiang2023bulk}, existence and uniqueness of localised interface modes in a band gap is proved. Specifically, we have the following result.

\begin{theorem}\label{thm:bulksum} 
    If
    \begin{align*}
        \mathcal{J}_{A} + \mathcal{J}_{B} \ne 0,
    \end{align*}
    then no interface mode exists. If 
    \begin{align*}
        \mathcal{J}_{A} + \mathcal{J}_{B} = 0,
    \end{align*}
    then there exists a unique frequency $\w_{r}\in\tilde{\mathfrak{A}}^U$, for which an interface mode exists.
\end{theorem}

\subsection{Complex frequencies}\label{subsection:RC}\hfill

Let us now view the problem of finding localised eigenmodes of the undamped system as a root-finding problem for $\w$ in the complex plane. Since the bands are real, the band gap region $\mathfrak{A}^U$ consists of the entire complex plane $\mathbb{C}$ except countably many intervals on the real axis. As established by the following results, there are no additional interface modes with frequency away from the real axis.
\begin{theorem}\label{thm:RC}
	The surface impedance $Z^{(U)}$ has no roots in $\mathbb{C}\setminus \mathbb{R}$.
\end{theorem}

\begin{proof}
	Let us assume that there exists $\w_{c}\in\mathbb{C}\setminus \mathbb{R}$ for which $Z^{(U)}=0$. Since the permittivity function $\ve$ is real-valued and positive, the operators $\mathcal{L}_j$, defined in \eqref{def:L} for permittivity $\ve_j$, is self-adjoint and positive-semidefinite. Hence, it admits only real, non-negative eigenvalues. Using the Helmholtz formulation in \eqref{eq:Helmholtz}, this translates to $\w_{c}^2\geq 0$, so that $\w_{c}^2 \in \R$. This concludes the proof.
\end{proof}

Let us also prove the following lemma which will be of great importance later.

\begin{lemma}\label{lemma:Z_2 no poles}
    The interface impedance function $Z^{(U)}$ has no poles in $\mathbb{C}\setminus \mathbb{R}$.
\end{lemma}

\begin{proof}
    For $Z^{(U)}$ to have a pole, it means that there exists $\w_p\in\mathbb{C}\setminus \mathbb{R}$ such that
    \begin{align*}
        \frac{\partial}{\partial x}u_A(x_0^-,\w_p)=0 \quad \text{or} \quad \frac{\partial}{\partial x}u_B(x_0^-,\w_p)=0,
    \end{align*}
    where $u_A$ and $u_B$ are defined as in Definition \ref{def: interf imp}. Let us assume that $\frac{\partial}{\partial x}u_A(x_0^-,\w_p)=0$.
    This implies that the boundary value problem
    \begin{align}\label{eq:boundaryvalueproblem}
    \left\{
    \begin{aligned}
        &\mathcal{L}_{{A}} u = \w_p^2 u, \quad x\in(-\infty,x_0],\\
        &\frac{\partial}{\partial x} u_{{A}}(x_0^-) = 0,\\
        &u_A, \frac{\partial}{\partial x} u_{{A}} \to 0 \text{ as } x\to -\infty,
    \end{aligned}
    \right.
    \end{align}
    admits a solution, where the differential operator $\mathcal{L}_{{A}}$ is given by \eqref{def:L} for permittivity function $\ve_A$. However,  the differential operator $\mathcal{L}_{{A}}$ is self-adjoint with respect to the inner product \eqref{def:inner product} on the space of functions satisfying the boundary conditions of \eqref{eq:boundaryvalueproblem} (Appendix \ref{app:boundary pb}). Hence, it admits only real eigenvalues which implies that $\w_p^2\in\mathbb{R}$. Since, in general, we consider frequencies with non-zero real parts, this implies that $\mathrm{Im}(\w_p)=0$, which is a contradiction. The case $Z^{(U)}_{B}=\infty$ follows the exact same reasoning. This concludes the proof.
\end{proof}

Combining this result with the dependence of $u$ on the coefficients of \eqref{eq:Helmholtz}, we obtain the following:

\begin{lemma}\label{lemma:case 2:Z hol}
    The interface impedance function $Z^{(U)}$ is holomorphic in $\mathbb{C}\setminus \mathbb{R}$.
\end{lemma}

\section{Damped systems}\label{sec:Damped}

Let us now move on to the case of complex-valued permittivity functions $\ve_j$, for $j=A,B$, with damping represented by a small imaginary part. More precisely, we introduce an arbitrary parameter $0<\d\ll1$ and we assume that
\begin{align}\label{limiting behaviour of damping}
    \lim_{\d\to0} \sup_{x\in\R} |\mathrm{Im}(\ve_A(x))| = \lim_{\d\to0} \sup_{x\in\R} |\mathrm{Im}(\ve_B(x))| = 0.
\end{align}
Essentially, by considering small imaginary parts for $\ve_A$ and $\ve_B$, we view the case of complex permittivities as a perturbation of the case of real permittivities. Indeed, let us denote by $Z^{(D)}_j$, $j=A,B$, the impedance functions of materials $A$ and $B$, respectively, and let $Z^{(D)}$ be the associated interface impedance. Then,
\begin{align}\label{limit Z3 to Z2}
    \lim_{\d\to0} Z^{(D)}_{j}(\w) = Z^{(U)}_{j}(\w),
\end{align}
for a fixed frequency $\w\in\mathbb{C}$ and for $j=A,B$.

We denote by $\mathfrak{A}^{D}_{A}$ and $\mathfrak{A}^{D}_{B}$ two band gaps of $\mathcal{L}_A$ and $\mathcal{L}_B$, respectively, in this regime, such that  
$$
\mathfrak{A}^{D} := \mathfrak{A}^{D}_{A} \cap \mathfrak{A}^{D}_{B} \ne \emptyset.
$$

\begin{lemma}\label{lemma:case 3:Z_j hol}
    The impedance functions $Z_{A}^{(D)}$ and $Z_{B}^{(D)}$ have no poles in the interior of $\mathfrak{A}^{D}$. 
\end{lemma}

\begin{proof}
    We treat the case of $Z_{A}^{(D)}$. Then, in the same reasoning, the result for $Z_{B}^{(D)}$ will follow. We denote by $(\mathfrak{A}^{D})^\circ$ the interior of $\mathfrak{A}^D$. 
    
    Let us assume that there exists $\w_p\in(\mathfrak{A}^{D})^\circ$ such that
    \begin{align*}
        Z_{A}^{(D)}(\w_p)=\infty.
    \end{align*}
    Then, for all $\ve_p>0$, there exists $r>0$ such that, for all $\w\in(\mathfrak{A}^{D})^\circ$ satisfying $|\w-\w_p|<r$, it holds that
    \begin{align*}
        \Big|Z_{A}^{(D)}(\w)\Big| > \ve_p.
    \end{align*}
    Then, applying the triangle inequality, we get
    \begin{align*}
        \Big|Z^{(U)}_{A}(\w)\Big| + \Big|Z_{A}^{(D)}(\w)-Z^{(U)}_{A}(\w)\Big| > \Big|Z_{A}^{(D)}(\w)\Big| > \ve_p.
    \end{align*}
    Now, letting our arbitrary parameter $\d\to0$, from \eqref{limit Z3 to Z2}, we have
    \begin{align*}
        \lim_{\d\to0} \Big|Z_{A}^{(D)}(\w)-Z^{(U)}_{A}(\w)\Big| = 0,
    \end{align*}
    which gives
    \begin{align*}
        \Big|Z_{A}^{(U)}(\w)\Big| > \ve_p.
    \end{align*}
    This translates to $Z^{(U)}_{A}$ having a pole at $\w_p$, which is a contradiction from Lemma \ref{lemma:Z_2 no poles}. This concludes the proof.
\end{proof}

Combining this with the dependence of a solution $u$ on the coefficients of \eqref{eq:Helmholtz}, we have:

\begin{lemma}
    The impedance functions $Z_{A}^{(D)}$ and $Z_{B}^{(D)}$ are holomorphic in the interior of $\mathfrak{A}^{D}$.
\end{lemma}

Applying Lemma~\ref{lemma:case 3:Z_j hol} on \eqref{def:interface impedance}, we get the following result.

\begin{corollary}\label{cor:Z_3 hol}
    The interface impedance $Z^{(D)}$ is holomorphic in the interior of $\mathfrak{A}^{D}$.
\end{corollary}

This leads to the main result for damped systems.

\begin{theorem}\label{thm:CC}
     Let $\mathfrak{A}^U$ and $\mathfrak{A}^D$ be the spectral band gaps in the undamped and damped cases, respectively.
     Assume that $Z^{(U)}$ has a root in $\mathfrak{A}^U\cap\mathfrak{A}^D$. Then, for small $\d$, $Z^{(D)}$ has a unique root in $\mathfrak{A}^D$ which converges to the root of $Z^{(U)}$ as $\d \to 0$.
\end{theorem}

\begin{proof}
    Since we are in the case of complex frequencies, the spectral band gaps of the operators $\mathcal{L}$ are defined as the complement of the spectral bands. We know that there exists $\w_U\in\mathfrak{A}^U\cap\mathfrak{A}^D$ such that $Z^{(U)}(\w_U)=0$. Now, let us define the set $N_\r$ to be
    \begin{align}\label{def:N_r}
        N_{\r} := \Big\{ \w\in\mathfrak{A}^U\cap\mathfrak{A}^D: \ |\w-\w_U|<\r \Big\}.
    \end{align}
	For small enough $\delta$ we have that $ \subset \Big(\mathfrak{A}^U\cap\mathfrak{A}^D\Big)^\circ$, where $X^\circ$ denotes the interior of the set $X\subset \mathbb{C}$. Then we have that for any $\f>0$ there exists $\r>0$ such that 
    \begin{align*}
         \Big| Z^{(U)}(\w) \Big| < \f,
    \end{align*}
    for all $\w\in N_{\r}$.
    From Lemma~\ref{lemma:Z_2 no poles} and Corollary~\ref{cor:Z_3 hol}, we get that
    \begin{center}
        $Z^{(D)}$ and $Z^{(U)}$ are holomorphic in $N_{\r}$.
    \end{center}
    In addition, from \eqref{limit Z3 to Z2}, we know that
    \begin{align*}
        \lim_{\d\rightarrow0}\Big|Z^{(D)}(\w)-Z^{(U)}(\w)\Big| = 0, \quad \text{for all } \w\in\overline{\mathfrak{A}^U\cap\mathfrak{A}^D},
    \end{align*}
    where $\overline{\mathfrak{A}^U\cap\mathfrak{A}^D}$ denotes the closure of $\mathfrak{A}^U\cap\mathfrak{A}^D$. It follows that, for small $\d$,
    \begin{align}\label{eq:Rouche}
        \Big|Z^{(D)}(\w)\Big| > \Big|Z^{(D)}(\w)-Z^{(U)}(\w)\Big|, \quad \text{for all } \w\in \partial N_{\r},
    \end{align}
    since, for all $\k>0$, 
    \begin{align*}
        \Big|Z^{(D)}(\w)-Z^{(U)}(\w)\Big| < \k \quad \text{ and } \quad \Big|Z^{(D)}(\w)\Big| \ne 0, \quad  \text{for all } \w\in \partial N_{\r},
    \end{align*}
    where $\partial N_{\r}$ denotes the boundary of $N_{\r}$.
    Thus, we can apply Rouch\'e's Theorem \cite[Chapter 1]{ammari2009layer} and obtain that $Z^{(D)}$ and $Z^{(U)}$ have the same number of roots in $N_{\r}$.
\end{proof}

\begin{remark}
    Here, let us note that Theorem \ref{thm:CC} can also be applied between two systems with dampings sufficiently close to each other. The proof follows similar steps and is given in Appendix \ref{app:cor:two dampings}.
\end{remark}

\section{Asymptotic behaviour} \label{sec:Asymptotic behaviour}

Our objective in this section is to obtain the asymptotic behaviour of the modes as $|x|\to\infty$. For this, we will make use of the transfer matrix method associated to this problem.

\subsection{Transfer matrix formulation}\hfill

Let us define
\begin{align}
    \tilde{u}(x) := \Big( u(x) , u'(x) \Big)^\top.
\end{align}

The symmetry we impose on the system is the following:
\begin{align*}
    \ve(x_n + h,\w) = \ve(x_{n+1}-h,\w), \ \ h\in[0,1),\  n\in\mathbb{N}\setminus\{0\}.
\end{align*}

For $j=A,B$, we denote the transfer matrix associated to the periodic cell of material $j$ by $T_p^{(j)}$, as described in \cite{alexopoulos2023topologically}. The following result holds. 
\begin{lemma}\label{lem:lambda}
    The transfer matrix $T_p^{(j)}$, for $j=A,B$, satisfies $\det(T_p^{(j)})=1$.
\end{lemma}

\begin{proof}
	This can be shown by following the procedure of transfer matrix method, as described in Section 5.1 in \cite{alexopoulos2023topologically}.
\end{proof}
 
The transfer matrix is used to propagate the solution across one unit cell. We have that
\begin{align}\label{eq:transferperiod}
    \begin{pmatrix}
        u(x_{n+1}) \\ u'(x_{n+1})
    \end{pmatrix}
    = \mathbbm{1}_{\{n\geq0\}} T_p^{(B)}
    \begin{pmatrix}
        u(x_{n}) \\ u'(x_{n})
    \end{pmatrix}
    + \mathbbm{1}_{\{n<0\}} T_p^{(A)}
    \begin{pmatrix}
        u(x_{n}) \\ u'(x_{n})
    \end{pmatrix},
\end{align}
for $n\in\mathbb{Z}$. From the symmetry of the model, we also get
\begin{align}\label{eq:symmetrycond}
    \begin{pmatrix}
        u(x_{n-1}) \\ u'(x_{n-1})
    \end{pmatrix}
    = \mathbbm{1}_{\{n\geq0\}} S T_p^{(B)} S
    \begin{pmatrix}
        u(x_{n}) \\ u'(x_{n})
    \end{pmatrix}
    + \mathbbm{1}_{\{n<0\}} S T_p^{(A)} S
    \begin{pmatrix}
        u(x_{n}) \\ u'(x_{n})
    \end{pmatrix},
\end{align}
for $n\in\mathbb{Z}$, where
\begin{align*}
    S = 
    \begin{pmatrix}
        1 & 0 \\
        0 & -1
    \end{pmatrix}.
\end{align*}
Here, let us note that $S=S^{-1}$.

Using the transfer matrix, we can easily obtain the band structure of the two periodic materials $A$ and $B$. Applying the quasiperiodic boundary conditions, we get
\begin{align*}
    \tilde{u}(x_{n+1}) = e^{i\k} \tilde{u}(x_n).
\end{align*}
Combining this with \eqref{eq:transferperiod}, we get the following problem
\begin{align}\label{eq:one before dispersion relation}
    \Big( T^{(j)}_p(\w) - e^{i\k} I \Big) \tilde{u}(x_n) = 0, \quad n\in\mathbb{N},
\end{align}
where $j=A,B$. From \eqref{eq:one before dispersion relation}, we obtain the dispersion relation for each material given by
\begin{align}\label{eq:dispersion_relation}
    \det\Big(T^{(j)}_p(\w) - e^{i\k} I\Big) = 0,\quad j=A,B.
\end{align}
The dispersion relation relates the quasiperiodicity $\k$ with the frequency $\w$ of the system and gives the structure of the spectral bands and the spectral gaps of \eqref{eq:Lj}-\eqref{eq:FB} for materials $A$ and $B$.

From Lemma \ref{lem:lambda}, we know that the eigenvalues of the transfer matrix are either both on the unit circle, or precisely one is inside the unit circle. Since the first case corresponds to a spectral band, we have the following result.
\begin{lemma}
     For $\w$ in a spectral band gap, the transfer matrix $T^{(j)}_p$ has eigenvalues denoted by $\l^{(j)}_1$ and $\l^{(j)}_2$, satisfying $|\l^{(j)}_1|<1$ and $|\l^{(j)}_2|>1$.
\end{lemma}

\subsection{Asymptotic behaviour of damped systems}\hfill

We will use the transfer-matrix method to obtain information about the asymptotic behaviour of the modes $u(x_n)$ as $n\to\pm\infty$.  Arguing in the same was as in \cite{alexopoulos2023topologically, davies2022symmetry}, we have the following results.

\begin{theorem}\label{thm:AB:spectral}
    The eigenfrequency $\w$ of a localised eigenmode of the Helmholtz problem \eqref{eq:Helmholtz} posed on a medium constituted by two different materials with damping must satisfy
    \begin{align}
        \begin{pmatrix}
            -V_{21}^{(B)}(\w) & V_{11}^{(B)}(\w)
        \end{pmatrix}
        \begin{pmatrix}
            V_{11}^{(A)}(\w) \\ -V_{21}^{(A)}(\w)
        \end{pmatrix}
        = 0,
    \end{align}
    where $(V_{11}^{(A)}(\w) , V_{21}^{(A)}(\w))^{\top}$ is the eigenvector of the transfer matrix $T_p^{(A)}$ associated to the eigenvalue $|\l_1^{(A)}|<1$ and $(V_{11}^{(B)}(\w) , V_{21}^{(B)}(\w))^{\top}$ is the eigenvector of the transfer matrix $T_p^{(B)}$ associated to the eigenvalue $|\l_1^{(B)}|<1$.
\end{theorem}

From this, we obtain the asymptotic behaviour of the localised eigenmodes as $n\to\pm\infty$.

\begin{corollary}\label{cor:asymptotic behaviour}
    A localised eigenmode $u$ of \eqref{eq:Helmholtz}, posed on a medium constituted by two semi-infinite arrays of different materials with damping, and its associated eigenfrequency $\w$ must satisfy 
    \begin{align*}
        u(x_n) = O\Big( |\l^{(A)}_1(\w)|^{|n|} \Big) \ \ \text{ and } \ \ u'(x_n) = O\Big( |\l^{(A)}_1(\w)|^{|n|} \Big) \ \ \text{ as } n\to-\infty,
    \end{align*}
    \begin{align*}
        u(x_n) = O\Big( |\l^{(B)}_1(\w)|^{|n|} \Big) \ \ \text{ and } \ \ u'(x_n) = O\Big( |\l^{(B)}_1(\w)|^{|n|} \Big) \ \ \text{ as } n\to+\infty,
    \end{align*}
    where $\l^{(j)}_1$, for $j=A,B$, are the eigenvalues of $T^{(j)}_p$ satisfying $|\l^{(j)}_1(\w)|<1$.
\end{corollary}

\section{Damping increments and large damping}\label{sec:Large damping}
In order for Theorem \ref{thm:CC} to hold, the magnitude of the damping must be sufficiently small so that   \eqref{eq:Rouche} holds. In fact, this bound is deeply rooted in the structure of the material and its parameters. 
For structures with larger damping, one might consider a sequence of structures with damping increments chosen so that \eqref{eq:Rouche} holds in each successive step. In order to choose the increments, we must look at \eqref{eq:Rouche}, where the impedances $Z^{(D)}$ and $Z^{(U)}$ appear. From the definition of $Z$ in \eqref{def:surface impedances}, we see the dependence on the localised mode $u$. Here, equation \eqref{eq:transferperiod} gives us the continuous dependence of $u$ on the material parameters and hence, on the damping. This makes choosing an appropriate damping increment a complicated task, as different materials will have a different effect on this choice.


Here, let us also note that Theorem \ref{thm:CC} can also be adapted to the setting of two dampings. Let us consider two permittivity functions $\ve^{(D_1)}$ and $\ve^{(D_2)}$, such that 
\begin{align}\label{two dampings limit}
	\lim_{\d\to0} \mathrm{Im}\Big( \ve^{(D_2)} \Big)(\d) = \mathrm{Im}\Big( \ve^{(D_1)} \Big),
\end{align}
and let us denote by $Z^{(D_1)}, \mathcal{L}^{(D_1)}$ and $Z^{(D_2)},\mathcal{L}^{(D_2)}$ the associated interface impedance functions and differential operators, respectively. Then, we have the following corollary of Theorem \ref{thm:CC}.

\begin{corollary}\label{cor:two dampings}
    Let $\mathfrak{A}^{D_1}$ be a band gap of $\mathcal{L}^{(D_1)}$ which contains a root of $Z^{(D_1)}$. Then, for small $\d$, $Z^{(D_2)}$ has a root, which converges to the one of $Z^{(D_1)}$ as $\d\to 0$.
\end{corollary}

\begin{proof}
    The proof is almost identical to Theorem \ref{thm:CC} and is given in Appendix \ref{app:cor:two dampings}.
\end{proof}

It is evident that the issue of the choice on the damping increment also applies on the case of Corollary \ref{cor:two dampings}. We show this numerically by considering an example of a mirror symmetric material, as described in Section \ref{sec:Numerical example} and depicted in Figure \ref{fig:config}. We define the regions $R_c$ and $R_w$ to be the regions on which the inequality of Rouch\'e's theorem, \emph{i.e.} \eqref{eq:Rouche} and \eqref{eq:Rouche for dampings}, can and cannot be applied, respectively. More precisely,
\begin{align}
    R_c := \Bigg\{ \w\in\mathfrak{A}^{D_1}: \quad \Big|Z^{(D_2)}(\w)\Big| > \Big|Z^{(D_2)}(\w)-Z^{(D_1)}(\w)\Big| \Bigg\}
\end{align}
and
\begin{align}
    R_w := \Bigg\{ \w\in\mathfrak{A}^{D_1}: \quad \Big|Z^{(D_2)}(\w)\Big| \leq \Big|Z^{(D_2)}(\w)-Z^{(D_1)}(\w)\Big| \Bigg\}.
\end{align}

In Figure \ref{fig:Rouche}, we have two cases with fixed dampings close to each other, and we observe how the roots of the interface impedance are placed with respect to the regions $R_c$ and $R_w$. In particular, let $\w_1$ denote the root of $Z^{(D_1)}$ and $\w_2$ denote the associated root of $Z^{(D_2)}$. Then, we see that for $\w=\w_1$, we are at the equality on Rouch\'e's inequality, which implies that $\w_1\in\partial R_{w}$, as shown in Figure \ref{fig:Rouche}. Then, for $\w=\w_2$, it is clear that the inequality \eqref{eq:Rouche for dampings} no-longer holds, and $\w_2\in R_{w}^{\circ}$. In both cases, we are able to enclose the two roots $\omega_1$ and $\omega_2$ by a curve entirely lying in the region $R_c$ where the inequality holds.

\begin{figure}[tbh]
    \begin{subfigure}[b]{0.5\linewidth}
        \includegraphics[width=\textwidth]{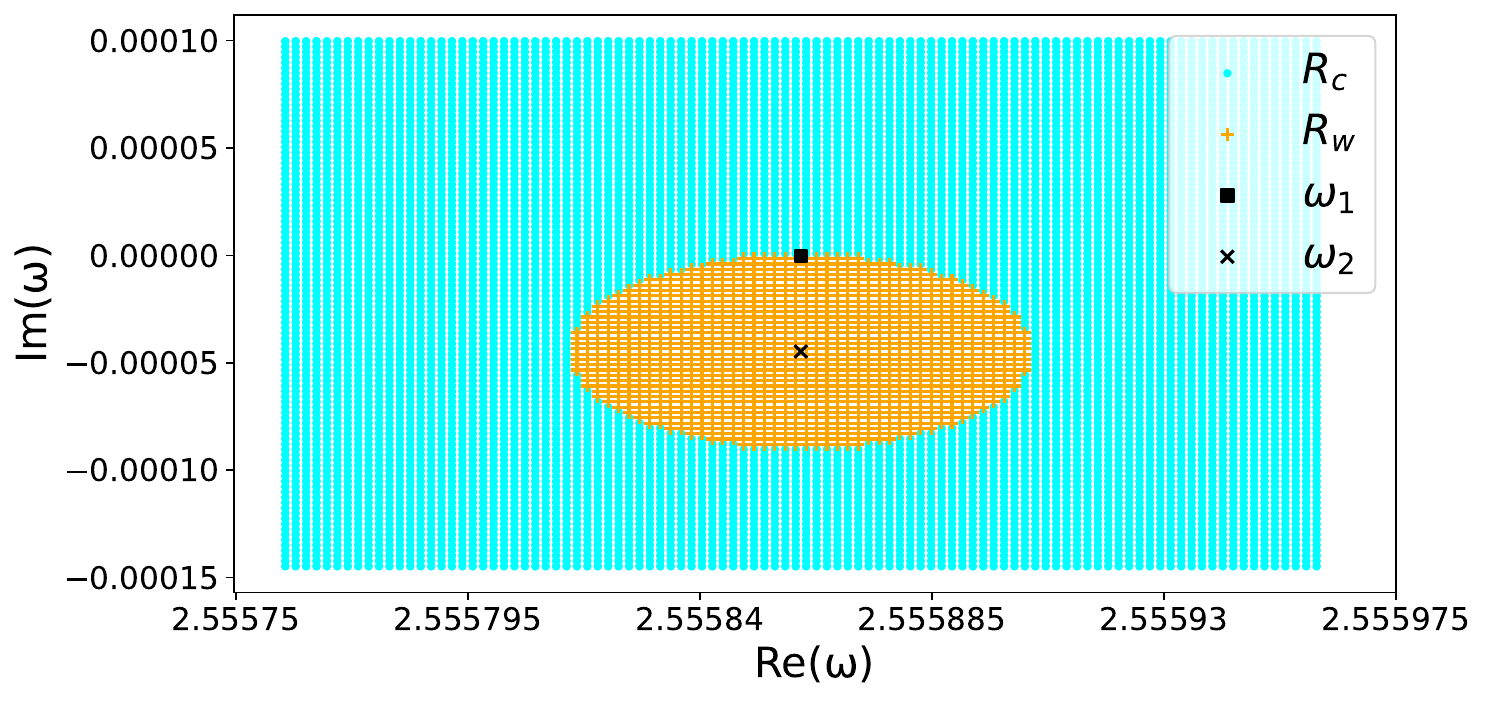}
        \caption{Dampings $0$ and $0.00005$.}
    \end{subfigure}\hfill
    \begin{subfigure}[b]{0.5\linewidth}
        \includegraphics[width=\textwidth]{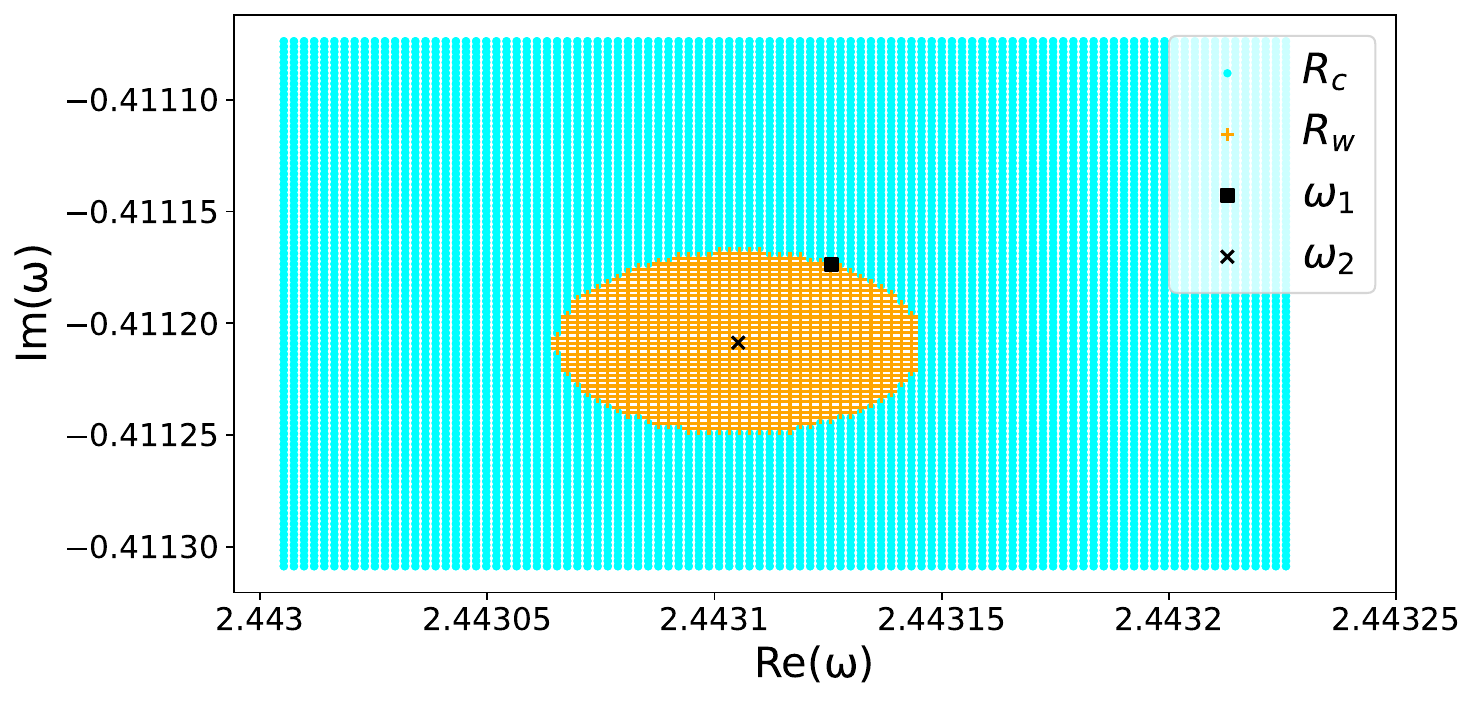}
        \caption{Dampings $0.5$ and $0.50005$.}
    \end{subfigure}
\caption{Numerically computed regions $R_c$ and $R_w$, where the inequality at the heart of Rouch\'e's theorem does and does not apply, respectively. The roots of $Z^{(D_1)}$ and $Z^{(D_2)}$ are shown. The root $\w_1$ of $Z^{(D_1)}$ (which has the smaller damping) must always lie on the interface between the two sets (\emph{i.e.} $\w_1\in \partial R_w$), while $\w_2\in R_w$. In these cases, both $\w_1$ and $\w_2$ can be enclosed by a closed curve entirely lying in $R_c$, meaning that Rouch\'e's theorem can be applied.} 
\label{fig:Rouche}
\end{figure}

\section{Numerical examples} \label{sec:Numerical example}

We now consider an example of a damped system with piecewise constant $\ve$, which can be thought of as a layered medium. In this case, we directly compute the interface modes of the system and illustrate our main results.

\subsection{Configuration}\hfill

Let us consider two materials $A$ and $B$, each one in the form of a semi-infinite array, glued at the interface $x_0$. Each unit cell is the product of layering; it contains three particles of two different permittivities $\ve_1$ and $\ve_2$. Each permittivity function satisfies the assumptions of Section \ref{subsection:MS:Damped systems}. We denote by $D^{[1]}_i$, $i=1,2$ and by $D^{[2]}_j$, $j=1,2,$ the particles of the unit cell with permittivity $\ve_1$ and $\ve_2$, respectively. In particular, for the mirror symmetry to be valid, the particles obey the following ordering:
\begin{align*}
    D^{[1]}_1 \ - \ D^{[2]}_1 \ - \ D^{[1]}_2 \ \text{ for material A}
\end{align*}
and
\begin{align*}
    D^{[2]}_1 \ - \ D^{[1]}_1 \ - \ D^{[2]}_2 \ \text{ for material B}.
\end{align*}
A schematic depiction of this configuration is given in Figure \ref{fig:config}.

\begin{figure}
    \centering
    \begin{tikzpicture}
        \node[inner sep=0pt] (russell) at (-4.5,6.5) {\includegraphics[width=0.3\textwidth]{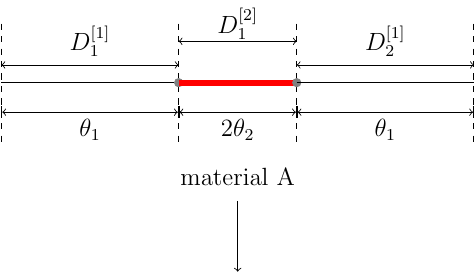}};
        \node[inner sep=0pt] (russell) at (4.5,6.5) {\includegraphics[width=0.3\textwidth]{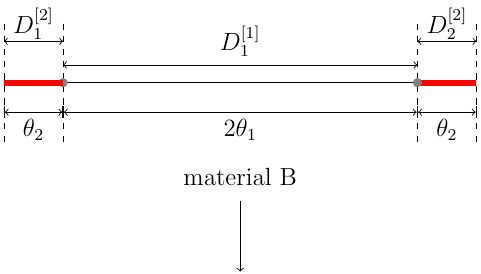}};
        \node[inner sep=0pt] (russell) at (0,4) {\includegraphics[width=1.0\textwidth]{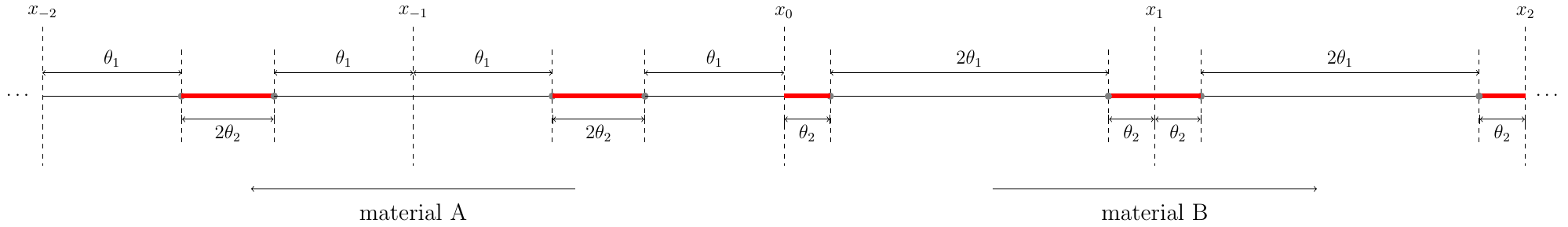}};
    \end{tikzpicture}
    \caption{An example of a damped system. In the unit cell of each material, we have 3 particles. We notice the mirror-symmetric way in which the particles are placed inside the periodic cells. Each material is constituted by a semi-infinite array created by periodically repeating the unit cell. Our structure is the result of gluing materials $A$ and $B$ at the interface $x_0$.} 
    \label{fig:config}
\end{figure}

In order to simplify our system, we will assume that the permittivities $\ve_1$ and $\ve_2$ are complex valued and constant on each particle. This translates to
\begin{align*}
    \ve(x) = 
    \begin{cases}
        \ve_1, & x \in D^{[1]} := \bigcup_{i=1}^{2} D^{[1]}_i,\\
        \ve_2, & x \in D^{[2]} := \bigcup_{i=1}^2 D^{[2]}_i,
    \end{cases}
\end{align*}
with $\ve_1,\ve_2\in\mathbb{C}$.
We consider an arbitrary parameter $\d>0$ and we assume that
\begin{align*}
    \mathrm{Im}(\ve_1) = c_1 \d \quad \text{ and } \quad \mathrm{Im}(\ve_2) = c_2\d, \quad c_1,c_2\in\mathbb{R}_{>0}.
\end{align*}
This allows us to vary the damping of the system by taking different values of $\d$.

\subsection{Band characterization}\hfill

The quasiperiodic Helmholtz problem \eqref{eq:Lj}-\eqref{eq:FB}, for each material, is characterised by a dispersion relation, as found in \eqref{eq:dispersion_relation}. This is an expression which relates the quasiperiodicity $\k$ with the frequency $\w$, such that it describes the spectral bands $\omega=\omega_n(\k)$. It is an equation of the form:
\begin{align}\label{eq:disp_rel}
    2\cos(\k) = f(\w),
\end{align}
where $f:\mathbb{R}\rightarrow\mathbb{C}$ is a function that depends on the material parameters and the system's geometry. This function is variously known as the discriminant or Lyapunov function of the operator $\mathcal{L}_j$ \cite{kuchment2016overview}. Some examples are provided in \emph{e.g.} \cite{alexopoulos2023effect, morini2018waves}.

From \eqref{eq:disp_rel}, we observe that, in order to be in a spectral band, the following two conditions need to be satisfied:
\begin{align}\label{band conditions}
    \mathrm{Im}(f(\w)) = 0 \quad \text{ and } \quad |f(\w)| < 2.
\end{align}

We will use these conditions to identify the spectral bands and the band gaps of our quasiperiodic Helmholtz problem, \emph{i.e.}:
\begin{itemize}
    \item Band: both condition in \eqref{band conditions} are satisfied.
    \item Band gap: at least one of the conditions in \eqref{band conditions} is not satisfied. 
\end{itemize}

\subsection{Effect of damping on spectrum} \hfill

A key feature of the proof given in Sections $\ref{sec:Undamped}$ and \ref{sec:Damped} is that, when we consider $\d=0$, \emph{i.e.}
\begin{align*}
    \mathrm{Im}(\ve_1) = \mathrm{Im}(\ve_2) = 0,
\end{align*}
we know that there exists a root $\w_U$ of $Z$ in the band gap, on the real axis.

In Figure \ref{fig:root localisation}, we observe how adding damping to the system affects the structure of the spectrum of \eqref{eq:Helmholtz}. By fixing the material parameters and increasing the damping of the system, \emph{i.e.} increasing $\d$, we observe a downwards shit of the spectral bands. In addition, we see how the location of the root of the interface impedance changes as the damping increases.

It is worth noting that Rouch\'e's theorem holds for small $\d$. In fact, as $\d$ increases, the inequality condition of the theorem fails and hence we cannot use this tool. Nevertheless, from the numerics, we observe that even for larger values of the damping, a root for $Z$ still exists in the corresponding band gap. Along with this, in Figure \ref{fig:root localisation}, we provide also three examples of fixed dampings which show how the spectral bands and the associated gaps look like. We see where the root of the interface impedance function $Z$ is located in each spectral gap. The key feature is that the edge mode shown to exist in the undamped case in \cite{coutant2024surface, thiang2023bulk, alexopoulos2023topologically} persists in the damped case.

\begin{figure}[tbh]
    \begin{subfigure}[t]{0.5\linewidth}
        \includegraphics[trim={0.2cm 0 1cm 0},clip,width=\textwidth]{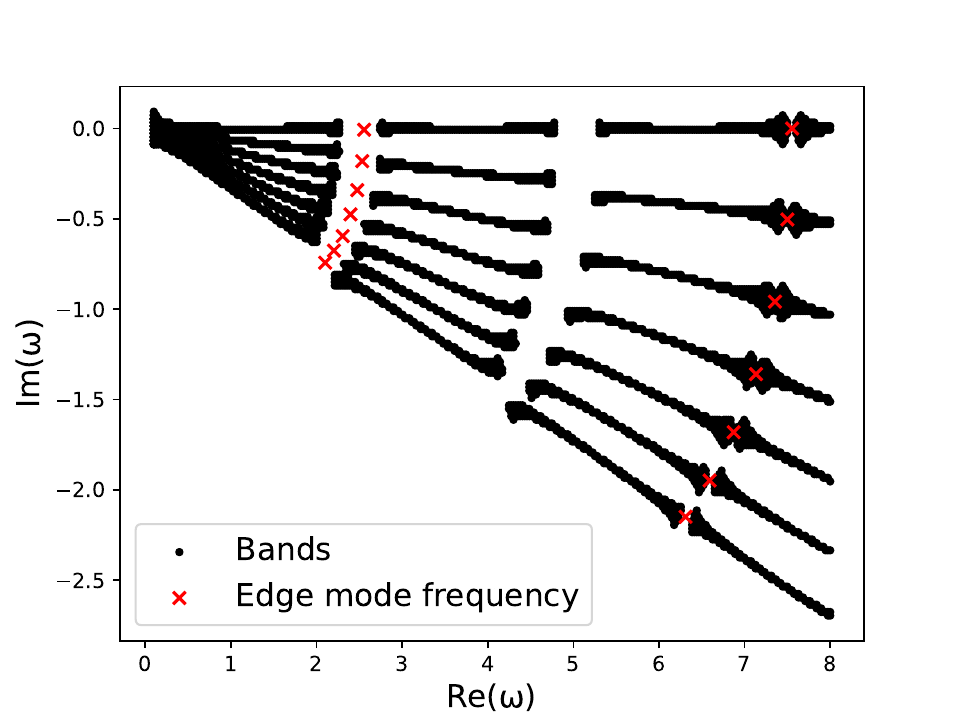}
        \caption{Band structure for varying damping.}
    \end{subfigure}\hfill
    \begin{subfigure}[t]{0.5\linewidth}
        \includegraphics[trim={0.5cm 0 1.5cm 0},clip,width=\textwidth]{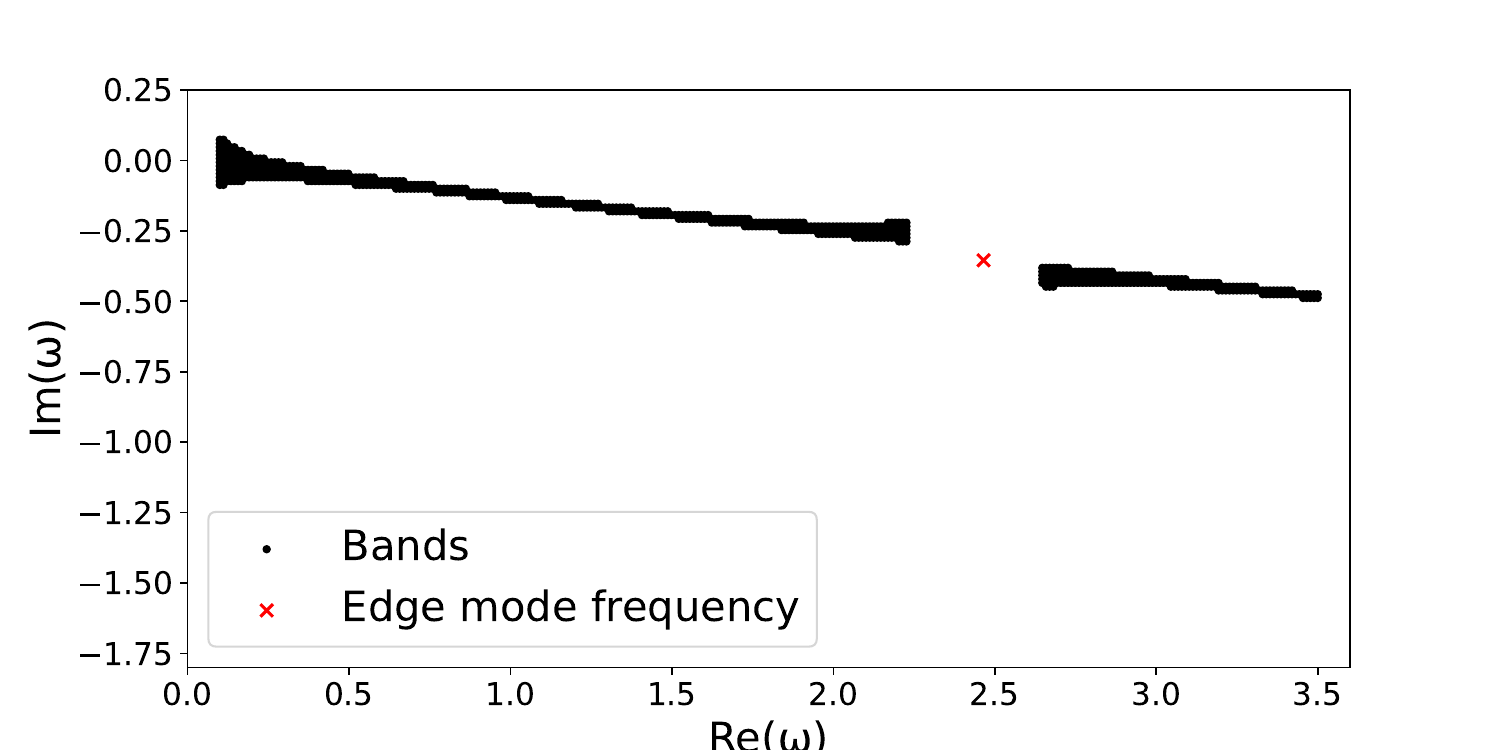}
        \caption{Band structure for damping 0.429.}
    \end{subfigure}\hfill
    \begin{subfigure}[t]{0.5\linewidth}
        \includegraphics[trim={0.5cm 0 1.5cm 0},clip,width=\textwidth]{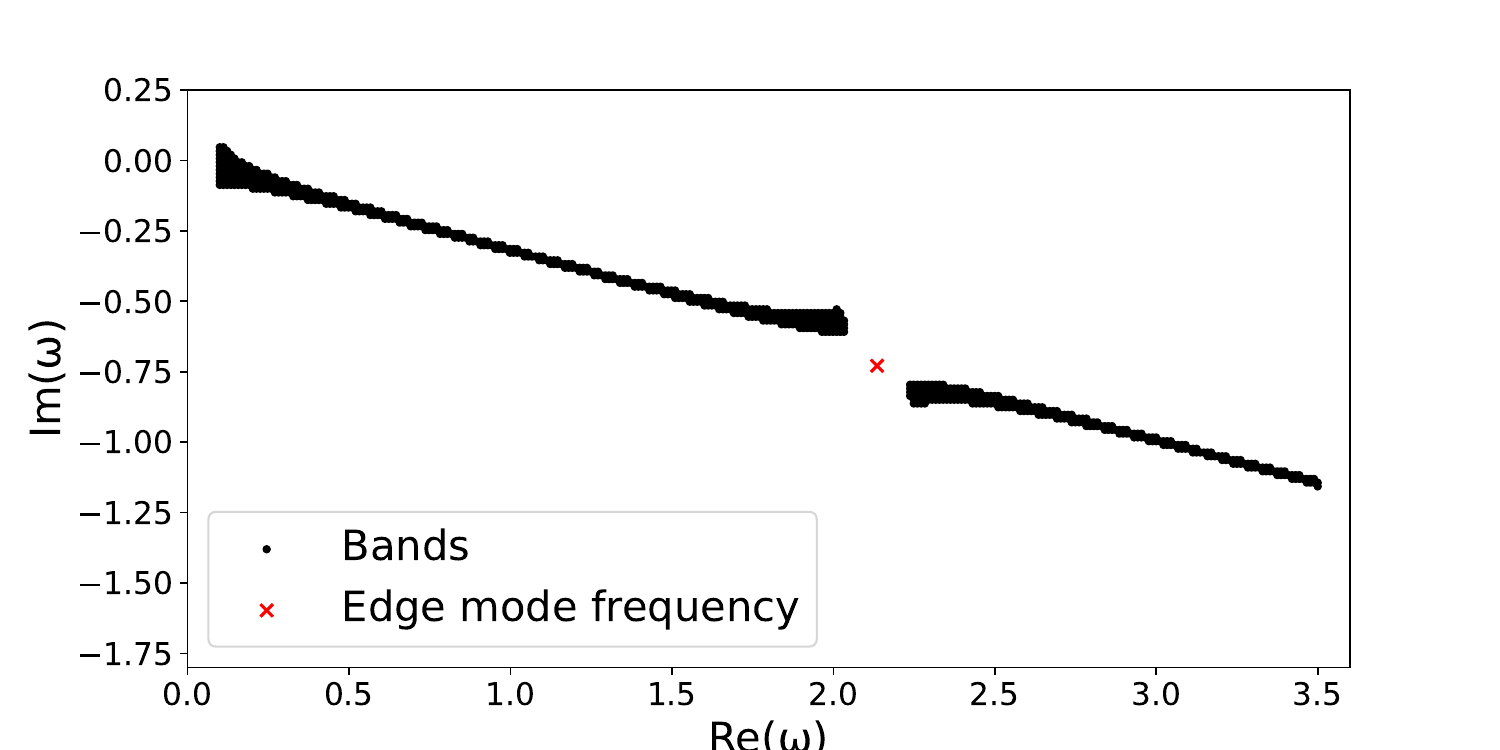}
        \caption{Band structure for damping 1.143.}
    \end{subfigure}\hfill
    \begin{subfigure}[t]{0.5\linewidth}
        \includegraphics[trim={0.5cm 0 1.5cm 0},clip,width=\textwidth]{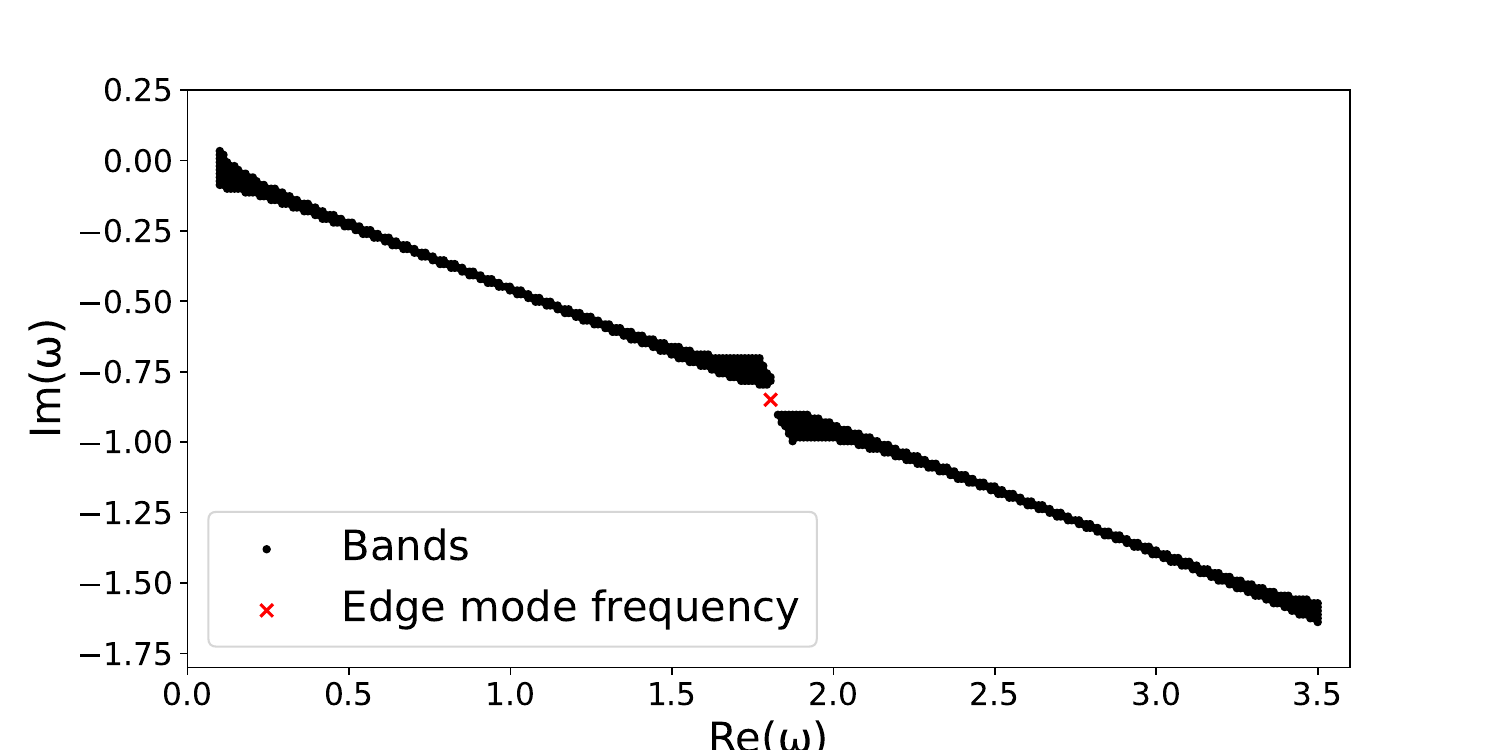}
        \caption{Band structure for damping 1.186.}
    \end{subfigure}
\caption{Spectral bands of the damped system studied in Section \ref{sec:Numerical example}. We observe how the increase in damping pushes the bands further away from the real axis. Also, we notice how the position of the roots of the interface impedance function $Z^{(D)}$ in a spectral gap changes as the damping changes. For three different and fixed values of damping, we provide the exact structure of the spectral bands.} 
\label{fig:root localisation}
\end{figure}

\subsection{Wave localisation} \hfill

Let us now  fix the damping and consider the edge state of the interface material. We take $\d>0$, so that $\mathrm{Im}(\ve_1),\mathrm{Im}(\ve_2)>0$. For this specific damping, we consider a frequency in a spectral band gap of \eqref{eq:Lj}-\eqref{eq:FB} for which the interface impedance function is zero, \emph{i.e.}, $Z^{(D)}(\w)=0$.

In Figure \ref{fig:eigenmode}, we observe the decaying character of $|u|$ as $|x|\to\infty$. The red lines denote the positions of the endpoints $x_n$'s. We see how $u$ oscillates at the scale of the unit cell, while its magnitude is localised around the interface $x_0$ and decays away from it. The blue curves denote the eigenvalue envelope given by Theorem \ref{cor:asymptotic behaviour}. In particular, denoting the eigenvalue envelope by $F$, we have:
\begin{align*}
    F(x_n) :=
    \left\{
    \begin{aligned}
        |\l^{(A)}_1(\w)|^{|n|}, &\quad n<0,\\
        |\l^{(B)}_1(\w)|^{|n|}, &\quad n>0,
    \end{aligned}
    \right.
\end{align*}
where $\l^{(A)}_1$ is the eigenvalue of $T^{(A)}_p$ satisfying $|\l^{(A)}_1(\w)|<1$ and $\l^{(B)}_1$ is the eigenvalue of $T^{(B)}_p$ satisfying $|\l^{(B)}_1(\w)|<1$.
We observe how the eigenmode follows the decay rate of the eigenvalue envelope as $|x|\to\infty$. This is the expected behaviour indicated by Theorem \ref{cor:asymptotic behaviour}.




\begin{figure}
\begin{center}
\includegraphics[width=0.8\textwidth]{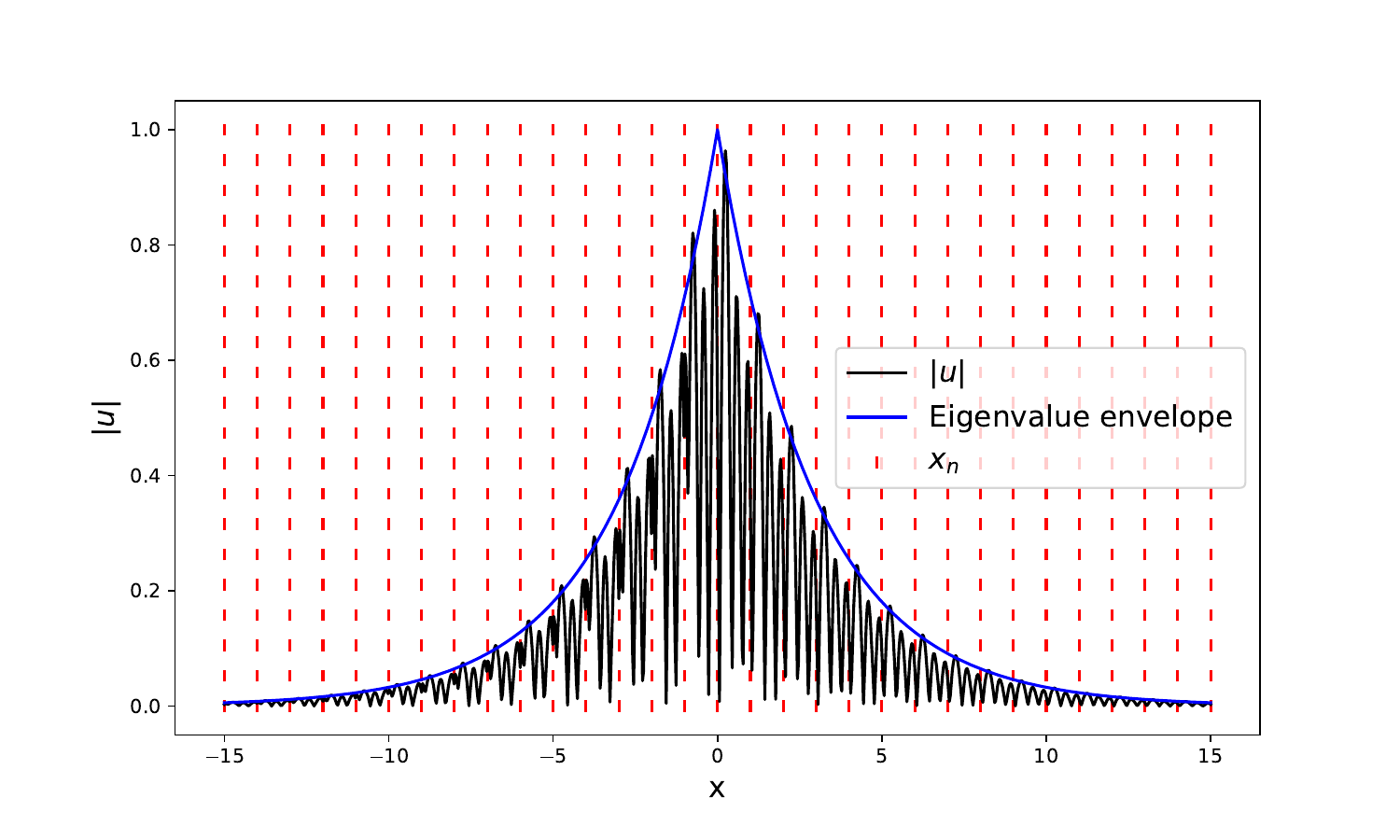}
\end{center}
\caption{For fixed damping, we see the behaviour of an interface localised mode as $|x|\rightarrow\infty$, for the damped system considered in Section \ref{sec:Numerical example}. The eigenvalue envelope shows the decay rates in terms of the eigenvalues of the transfer matrices, as described in Theorem \ref{cor:asymptotic behaviour}.} 
\label{fig:eigenmode}
\end{figure}

\section{Conclusion}
We have shown that localised interface modes exist for materials with damping. Specifically, we have shown that the interface modes previously studied for undamped materials persist when the imaginary parts of the  material parameters are non-zero. Viewing the damped system as a perturbation of the undamped one and using Rouch\'e's theorem, we are able to show the existence of localised modes for the damped model. These localised modes are unique in the sense that they correspond uniquely to the localised modes in the undamped system (and they coalesce in the limit of small damping). We have computed the edge mode frequency numerically and found that the edge mode exists even in the case of relatively large dampings. Finally, we adapted the transfer matrix method to the case of damped systems and obtained the asymptotic behaviour of the localised modes, as $|x| \rightarrow \infty$, in terms of the eigenvalues of the transfer matrix of each material. Since the quantisation of the bulk index is broken by the damping, proving  topological robustness of the herein studied interface modes remains an open problem.

\section*{Conflicts of interest}

The authors have no conflicts of interest to disclose.

\section*{Acknowledgements}

The work of KA was supported by ETH Z\"urich under the project ETH-34 20-2.

\appendix

\section{Self-adjointness in Lemma~\ref{lemma:Z_2 no poles}} \label{app:boundary pb}

We consider two functions $u, v$ satisfying the boundary conditions at infinity as described in \eqref{eq:boundaryvalueproblem}. We will show that
\begin{align*}
    \langle \mathcal{L}_A u, v \rangle = \langle u, \mathcal{L}_A v \rangle,
\end{align*}
where the inner product $\langle\cdot,\cdot\rangle$ is defined in \eqref{def:inner product}.

It holds
\begin{align*}
    \langle \mathcal{L}_A u, v \rangle &= \int_{-\infty}^{x_0} \frac{1}{\m_0} \frac{\partial}{\partial x}\left( \frac{1}{\ve(x)} \frac{\partial}{\partial x} u(x) \right) \overline{v}(x) \upd x \\
    &= \left[ \frac{1}{\m_0} \frac{1}{\ve(x)} \frac{\partial}{\partial x} u(x) \overline{v}(x) \right]_{-\infty}^{x_0} - \int_{-\infty}^{x_0} \frac{1}{\m_0} \frac{\partial}{\partial x} u(x) \frac{1}{\ve(x)} \frac{\partial}{\partial x} \overline{v}(x) \upd x\\
    &= - \int_{-\infty}^{x_0} \frac{1}{\m_0} \frac{\partial}{\partial x} u(x) \frac{1}{\ve(x)} \frac{\partial}{\partial x} \overline{v}(x) \upd x,
\end{align*}
where the last step follows since $\frac{\partial}{\partial x}u(x_0^{-}) = 0$ and $\frac{\partial}{\partial x}u(x) \to 0$ as $x\to-\infty$. Then,
\begin{align*}
    \langle \mathcal{L}_A u, v \rangle &= - \left[ \frac{1}{\m_0}  u(x) \frac{1}{\ve(x)} \frac{\partial}{\partial x} \overline{v}(x) \right]_{-\infty}^{x_0} +\int_{-\infty}^{x_0} u(x) \frac{1}{\m_0} \frac{\partial}{\partial x} \left( \frac{1}{\ve(x)} \frac{\partial}{\partial x} \overline{v}(x) \right) \upd x \\
    &= \int_{-\infty}^{x_0} u(x) \frac{1}{\m_0} \frac{\partial}{\partial x} \left( \frac{1}{\ve(x)} \frac{\partial}{\partial x} \overline{v}(x) \right) \upd x\\
    &= \langle u, \mathcal{L}_A v \rangle,
\end{align*}
since $u(x) \to 0$ as $x\to-\infty$ and $\frac{\partial}{\partial x}v(x_0^-) = 0$. This concludes the proof.

\section{Proof of Corollary~\ref{cor:two dampings}}\label{app:cor:two dampings}

\begin{proof}

Since we are in the case of complex frequencies, the spectral band gaps of the operators $\mathcal{L}$ are defined as the complement of the spectral bands. This implies that $\mathfrak{A}^{D_1} \cap \mathfrak{A}^{D_2} \ne \emptyset$. In addition, let us assume that there exists $\w_1\in\mathfrak{A}^{D_1}\cap\mathfrak{A}^{D_2}$ such that $Z^{(D_1)}(\w_1)=0$. Now, let us define the set $N_\r$ to be
    \begin{align}\label{def:N_r two dampings}
        N_{\r} := \Big\{ \w\in\mathfrak{A}^{D_1}\cap\mathfrak{A}^{D_2}: \ |\w-\w_1|<\r \Big\} \subset \Big(\mathfrak{A}^{D_1}\cap\mathfrak{A}^{D_2}\Big)^\circ,
    \end{align}
    such that for any $\f>0$ there exists $\r>0$ such that
    \begin{align*}
        \Big| Z^{(D_1)}(\w) \Big| < \f,
    \end{align*}
    for all $\w\in N_{\r}$. From  Corollary~\ref{cor:Z_3 hol}, we get that
    \begin{center}
        $Z^{(D_1)}$ and $Z^{(D_2)}$ are holomorphic in $N_{\r}$.
    \end{center}
    In addition, from \eqref{two dampings limit}, we know that
    \begin{align*}
        \lim_{\d\rightarrow0}\Big|Z^{(D_1)}(\w)-Z^{(D_2)}(\w)\Big| = 0, \quad \text{for all } \w \in \overline{ \mathfrak{A}^{D_1} \cap \mathfrak{A}^{D_2} }.
    \end{align*} 
    It follows that, for small $\d$,
    \begin{align}\label{eq:Rouche for dampings}
        \Big|Z^{(D_2)}(\w)\Big| > \Big|Z^{(D_2)}(\w)-Z^{(D_1)}(\w)\Big|, \quad \text{for all } \w\in \partial N_{\r},
    \end{align}
    since, for all $\k>0$, 
    \begin{align*}
        \Big|Z^{(D_2)}(\w)-Z^{(D_1)}(\w)\Big| < \k \quad \text{ and } \quad \Big|Z^{(D_2)}(\w)\Big| \ne 0, \quad  \text{for all } \w\in \partial N_{\r},
    \end{align*}
    where $\partial N_{\r}$ denotes the boundary of $N_{\r}$.
    Thus, we can apply Rouch\'e's Theorem \cite[Chapter 1]{ammari2009layer} and obtain that $Z^{(D_1)}$ and $Z^{(D_1)}$ have the same number of roots in $N_{\r}$.

\end{proof}

\bibliographystyle{abbrv}
\bibliography{references}{}

\end{document}